\documentclass[12pt]{article}
\usepackage{latexsym,amsmath,amsthm,verbatim,ifthen,amssymb}
\usepackage[all]{xy}

\newtheorem{theorem}{Theorem}
\newtheorem{corollary}[theorem]{Corollary}
\newtheorem{lemma}[theorem]{Lemma}
\newtheorem{proposition}[theorem]{Proposition}
\newtheorem{example}[theorem]{Example}

\newtheorem{remark}[theorem]{Remark}

\providecommand{\Sec}{\operatorname{\rm Sec}}

\providecommand{\map}{\operatorname{\rm map}}

\providecommand{\Hom}{\operatorname{\rm Hom}}

\providecommand{\cat}{\operatorname{\rm cat}}

 \newcommand{\cu }{\mathbb{Q}}

 \newcommand{\abs}[1]{\left\vert#1\right\vert}

\newcommand{\cdga}{{CDGA}}

\newcommand{\bz}{\mathbb Z}

\newcommand{\bs}{\mathbb S}
\newcommand{\bt}{\mathbb T}
\newcommand{\bc}{\mathbb C}
\newcommand{\bq}{\mathbb Q}
\newcommand{\bk}{\mathbb K}
\newcommand{\bF}{\mathbb F}

  \DeclareMathOperator{\base}{\mathfrak{B}}

\newcommand{\pushoutcorner}[1][dr]{\save*!/#1-1.2pc/#1:(-1,1)@^{|-}\restore}

\begin{document}

\title{The homotopy fixed point set of Lie group actions on  elliptic spaces}
\author{Urtzi Buijs\footnote{Partially supported by the \emph{MICINN} grant MTM2010-15831, by the  grants FQM-213, 2009-SGR-119, and by the Marie Curie COFUND programme U-mobility, co-financed by the University of M\'alaga, the European Commision FP7 under GA No. 246550, and Ministerio de Econom\'{\i}a y Competitividad (COFUND2013-40259).}, Yves F\'elix$^{\dag}$,$\,$ Sergio Huerta$^{\dag}$ and Aniceto Murillo\footnote{Partially supported by the {\em MICINN}  grant MTM2010-18089 and by the Junta de
Andaluc\'\i a grant FQM-213.\hfill\break 2010 M.S.C.: 55P62, 55R91, 55R45.\hfill\break Keywords: Fixed point set, Homotopy fixed point set, rational homotopy, space of sections.}\\ }
\maketitle

\begin{abstract}
Let $G$ be a  compact connected Lie group, or more generally a path connected topological group of the homotopy type of a finite CW-complex, and let $X$ be a rational nilpotent $G$-space. In this paper we analyze the  homotopy type of the homotopy  fixed point set $X^{hG}$, and the natural injection $k\colon X^G\hookrightarrow X^{hG}$. We show that if $X$ is elliptic, that is, it has finite dimensional rational homotopy and cohomology, then each path component of $X^{hG}$ is also elliptic. We also give an explicit algebraic model of the inclusion $k$ based on which we can prove, for instance, that  for $G$ a torus, $\pi_*(k)$ is injective in rational homotopy but, often, far from being a rational homotopy equivalence.
\end{abstract}

\section*{Introduction}

  In one of its versions, the  \emph{generalized Sullivan conjecture} \cite{carl,hae,lan} asserts that, whenever $G$ is a finite $p$-group and $X$ is a $G$-CW-complex, then the inclusion
$$
k\colon X^G\hookrightarrow X^{hG}
$$
of the {\em fixed point set} into the {\em homotopy fixed point set} induces an isomorphism in cohomology with $\bF_p$ coefficients. Equivalently, the $p$-completion of $k$,
$$
k^{\wedge}_p\colon (X^G)^{\wedge}_p\hookrightarrow (X^{hG})^{\wedge}_p,
$$
is a homotopy equivalence. In this paper, broadly speaking, we study this fundamental question in the non-discrete case and when $p=0$.

Recall that the homotopy fixed point set of a given $G$-action on $X$ is defined as the space $\map^G(EG,X)$ of equivariant maps  from the universal $G$-space $EG$ into $X$. In the same way, the inclusion $k$ can be seen as the map
$$\map^G(*,X)\to\map^G(EG,X)$$
 induced by $EG\to*$.

From now on, and unless explicitly stated otherwise, $G$ will denote a compact connected Lie group, or more generally a path connected topological group of the homotopy type of a finite CW-complex. In the same way, by a topological $G$-space  we mean   a nilpotent $G$-space of the homotopy type of a CW-complex of finite type.
Even though the action of $G$ on such a space $X$ induces an action in the classical rationalization $X_\bq$, the homotopy fixed point set of the resulting action $(X_\bq)^{hG}$ may fail to be nilpotent (see \cite[Ex. 15]{bufemu2}) in which case, unambiguous rationalization is not possible. Besides, $(X_\bq)^{hG}$  may not have the homotopy type of a CW-complex, and in this case,  algebraic models  would only recover its weak homotopy type.

We then start by  setting a  sufficiently general context in  which the homotopy fixed point set of $G$-actions in rational nilpotent spaces has the homotopy type of a nilpotent CW-complex. Identifying $X^{hG}$ with the space  $\sec \xi$ of sections of the Borel fibration
$$
X\to X_{hG}\stackrel{\xi}{\to} BG,
$$
we see that,  if $\pi_{>n}X$ is torsion for a certain $n>1$, in particular if $X$ is elliptic, then $(X_\bq)^{hG}$ is a rational nilpotent complex of the homotopy type of a CW-complex. Then we prove (see Theorem \ref{principal1}),

\begin{theorem}\label{introteo} If $X$ is an elliptic space, then each path component of the homotopy fixed point set $(X_\bq)^{hG}$ is also elliptic.
\end{theorem}

A consequence of this if is the following criterium to detect elliptic  $\bs^1$-spaces, see Corollary \ref{corolario1}.

\begin{corollary}\label{corolarioi}
A  $\bs^1$-space $X$ is elliptic if and only if each path component of the homotopy fixed point set $(X_\bq)^{h\bs^1}$ has finite Lusternik-Schnirelmann category.
\end{corollary}

As examples, we explicitly describe the rational homotopy type of the homotopy fixed point set of certain torus actions (see theorems \ref{accionenesfera}, \ref{torito} and \ref{complejito}).

\begin{theorem}\label{ejemplo1}
Let $\bt^m$ be the $m$-dimensional torus. The Borel fibration of any $\bt^m$-action with fixed points in an odd dimensional rational sphere $\bs_\bq^n$ is homotopically trivial. In particular,
$$
{\bs_\bq^n}^{h\bt^m}
\simeq
\map\bigl(({\bc P}^\infty\times\stackrel{m}{\cdots}\times{\bc P}^\infty)^{(n+1)} ,\bs^n_\bq\bigr)
,$$ which has the rational homotopy type of a finite product of odd dimensional spheres. For $m=1$ we can be more accurate:
$$
 {\bs_\bq^n}^{h\bs^1}\simeq \map\left({\bc P}^{\frac{n+1}{2}},\bs^n\right)_\bq \simeq_\bq \bs^1\times\bs^3\times\cdots\times\bs^n.
 $$
\end{theorem}
\begin{theorem}\label{ejemplo2}
Given circle actions with fixed points on the rational even sphere $\bs_\bq^n$ and the rational complex projective space $\bc P^n_\bq$,
\begin{itemize}
\item[(i)] The homotopy fixed point set ${\bs^n_\bq}^{h\bs^1}$ is either path component and of the rational homotopy type of
    $$
    \bs^1\times\frac{SO(n+2)}{U\left(\frac{n+2}{2}\right)},
$$
or else, it has two components, each of them of the rational homotopy type of
    $$
    \bs^{n+1}\times\bs^{n+3}\times\cdots\times\bs^{2n-1}.
     $$

\item[(ii)] The homotopy fixed point set ${\bc P^n_\bq}^{\,h\bs^1}$ has at most $n+1$ path components and each of them is of the rational homotopy type of one of the following spaces:
    $$
    \bs^1\times\bs^3\times\dots\times\bs^{2i-1}\times {\bc P}^i\times\bs^{2i+3}\times\dots\times\bs^{2n+1},\quad i=1,\dots,n.
    $$
\end{itemize}
\end{theorem}

Concerning the homotopical behaviour of the inclusion of the fixed point set into its homotopy counterpart, given any rational $G$-space $X$ for which $\dim \pi_*(X)<\infty$, in particular any rational elliptic space, we build first a Sullivan model of
$
k\colon X^G\hookrightarrow X^{hG}
$
(see Theorem \ref{principal2}). Then, for $G$ a torus $\bt$, we prove the following (see Theorem \ref{principal3}), which already appears in \cite{bufemu2} for $\bs^1$ actions.

\begin{theorem}\label{introteo2}
The morphism induced by
$k\colon X^{\bt}\hookrightarrow X^{h\bt}$
 in rational homotopy groups is injective. In particular,
  if both $X^\bt$ and $X^{h\bt}$ are simply connected, then
$
\cat X^\bt\le \cat X^{h\bt}$.
\end{theorem}
In certain cases we can be more explicit. For instance, when $X=\bs_\bq^n$ is an odd dimensional  rational sphere, the map
$$
k\colon X^\bt\hookrightarrow X^{h\bt}
$$
has the rational homotopy type of the inclusion of a certain factor
$$
S^{n_i}\hookrightarrow \bs^{n_1}\times\dots\times\bs^{n_p},
$$
in a product of odd dimensional spheres (see Proposition \ref{ejemplo}). Nevertheless, for $\bs^1$-actions,  $k$ often fails to be a rational homotopy equivalence, see Proposition \ref{nunca}:

\begin{proposition}\label{joder2} For any minimal $\bs^1$-action on a non trivial, simply connected, rational elliptic space $X$, the map
 $
k\colon X^{\bs^1}\hookrightarrow X^{h\bs^1}
$
is never a homotopy equivalence.
\end{proposition}

In the next section we set notation, and presents a summary of the results we shall use  concerning the rational homotopy type of  spaces of sections of nilpotent fibrations. In Section 3, we prove Theorem \ref{introteo}. Section 4 contains Corollary \ref{corolarioi} and Theorems \ref{ejemplo1} and \ref{ejemplo2}. Finally, in \S5, we find a model of the map $k$ and prove  Theorem \ref{introteo2} and Proposition \ref{joder2}.

\section{Preliminaries: spaces of sections and their models}

We shall heavily rely on the basics of rational homotopy theory, all of which can be found in   \cite{fehatho}. Here, we simply remark a few facts.

We begin by recalling what we mean by a  {\em model} of a non necessarily path connected space $X$ as long as all its components are nilpotent CW-complexes of finite type. For it, consider the pair of adjoint functors \cite{bousgu,su},
$$\xymatrix{ {\text{\bf SimpSet}}& {\text{\bf CDGA}} \ar@<1ex>[l]^{\langle\,\,\rangle}
\ar@<1ex>[l];[]^{A_{PL}}\\}
$$
between the homotopy categories of commutative differential graded algebras,   \cdga's henceforth, and simplicial sets. As in \cite{bumu1}, a {\em model} of a space $X$ as above, is a free CDGA $(\Lambda W,d)$, possibly $\mathbb Z$-graded, such that its
simplicial realization $\langle(\Lambda W,d)\rangle$ has the same
homotopy type as the Milnor simplicial approximation $S_*(X_{\mathbb Q})$ of the rationalization $X_\cu$ of $X$. Here $X_\bq$ denotes the space whose path components are the classical rationalizations of the path components of $X$. In the same way, a model for a map $f\colon X\to Y$ is a morphism $(\Lambda U,d)\to(\Lambda W,d)$ of free CDGA's such that its simplicial realization has the same homotopy type as the simplicial approximation $S_*(f_\bq)$ of the rationalization $f_\bq$.

We also recall that a space $X$ is said to be (rationally) {\em elliptic} if both $H^*(X;\bq)$ and $\pi_*(X)\otimes\bq$ are finite dimensional vector spaces.

A classical characterization of elliptic spaces in terms of their Sullivan models reads as follows:
Recall that a Sullivan algebra
 $(\Lambda V,d)$  is {\em pure} if
$
d V^{\text{even}}=0$ and $ d V^{\text{odd}}\subset \Lambda V^{\text{even}}$.
Given a Sullivan algebra $(\Lambda V,d)$, its {\em associated pure Sullivan algebra} $(\Lambda V,d_\sigma)$ is obtained by decomposing, for each $v\in V$, $dv=d_\sigma v + \Phi_v$ with $d_\sigma v\in \Lambda V^{\text{even}}$ and $\Phi_v\in \Lambda^+V^{\text{odd}}\otimes\Lambda V^{\text{even}}$. By degree reasons,  $d_\sigma V^{\text{even}}=0$.

\begin{theorem}{\em \cite[Prop. 1]{hal1}}\label{prop:puroEliptico}
Let $(\Lambda V,d)$ be a Sullivan algebra with $\dim V<\infty$. Choose a homogenous basis $\{v_1,\dots,v_n\}$ of $\,V$ for which $dv_k\in\Lambda(v_1,\dots,v_{k-1})$, $k=1,\dots,n$. Then, the following are equivalent.

\medskip

(i) $\dim H^*(\Lambda V,d)<\infty$.

\medskip

(ii)
 $\dim H^*(\Lambda V, d_\sigma)<\infty$.

\medskip

(iii) For each even element $v_k$  of the considered basis, the cohomology class  $[v_k]\in H^*(\Lambda (v_k,\dots,v_n),\overline d\,)$ is nilpotent, that is, there exists a positive integer $N_k$ and elements
$$\Psi_k\in\Lambda V,\ \ \Theta_k\in\Lambda^+(v_1,\dots,v_{k-1})\cdot\Lambda V$$
 such that
$$
d \Psi_k=v_k^{N_k}+\Theta_k\text{.}
$$

\end{theorem}

As most of our work rely in the identification of the homotopy fixed point set of a given action with the space of sections of the associated Borel fibration, we make a quick overview of the result we use concerning the rational homotopy type of the space of sections of a fibration. To do so  fix a nilpotent fibration,
$$F\to E\stackrel{p}{\to} B,$$
  that is, a fibration of  path connected, nilpotent spaces in which the action of $\pi_1(B)$ in the homotopy groups of the fibre is also nilpotent. By $\sec p$  we denote the space of continuous sections of $p$. If $\sigma$ is such a section, $\sec_\sigma p$ denotes the path component of $\sec p$ containing $\sigma$. In the pointed category these are $\sec^* p$ and $\sec^*_\sigma p$. Recall that the space $\map(X,Y)$ (respectively $\map^*(X,Y)$) of continuous maps (respectively pointed continuous maps) is simply the space of sections (respectively pointed sections) of the trivial fibration
$$Y\to X\times Y\to X.$$

It is well known  \cite[Chap. II, Thm. 2.6]{hilmisroit}, \cite[Thm. 4.1]{mo}, that if $B$ is a finite CW-complex and $F,E$ are CW-complexes of finite type, then $\sec p$ has the homotopy type of a CW-complex of finite type and each of its path component is nilpotent. In this case, it can be easily deduced from \cite[Thm. 5.3]{mo} that $\sec p_\bq$ has the weak homotopy type of $(\sec p)_\bq$,
$$
\sec p_\bq\simeq_w (\sec p)_\bq.
$$
However, this is not, in general, a homotopy equivalence as $\sec p_\bq$ may fail to be of the homotopy type of a CW-complex. Nevertheless,  $B_\bq$ is still of finite dimension and thus, the classical result \cite[Thm. 1.1]{kahn}, generalized to spaces of sections \cite{smre}, let us asserts that: if $\dim B=n$ and $\pi_{>n}(F)$ is torsion for some $n\ge 1$, then $\sec p_\bq$ is of the homotopy type of a CW-complex and therefore,
$$
\sec p_\bq\simeq (\sec p)_\bq.
$$
 If in the original fibration $F$ is already a rational space, it is important to remark that all of the above translates to the following: $\sec p$ is already a rational space of the weak homotopy type of $\sec p_\bq$ and it is homotopy equivalent to it whenever $B$ is finite dimensional and $\pi_{>\dim B}F=0$.

 We now turn to the case in which $B$ is not longer a finite complex. In this case, and for any $n\ge 1$, we denote by
 $$
 F\longrightarrow E_{n}\stackrel{p_n}{\longrightarrow}B^{(n)}
 $$
 the pullback fibration of the inclusion $B^{(n)}\hookrightarrow B$ of the $n$-skeleton over $p$,
$$ \xymatrix{
F\ar@{=}[r]\ar[d]&F\ar[d]\\
{\pushoutcorner{E_{n}}}\ar[d]^{p_n}\ar[r]&{E}\ar[d]^p\\
{B^{(n)}}\ar@{^(->}[r]&{B.}
}
$$
Observe that, if $\sigma\in\sec p$, the map $B^{(n)}\hookrightarrow B\stackrel{\sigma}{\to} E$ induces a section $\sigma_n\in\sec p_n$. This process defines  fibrations,
$$\xymatrix{
&&&\sec p{\qquad}\ar[lld]\ar[ld]\ar[rd]\ar[rrd]&&\\
\dots\ar[r]&\sec p_n\ar[r]&\sec p_{n-1}\ar[r]&\dots\ar[r]&\sec p_2\ar[r]&\sec p_1.\\
}
$$
The same applies to each path component and to the pointed case. It is immediate that,
$$
\begin{aligned}
&\sec p=\varprojlim\sec p_n,\\
&\sec_\sigma p=\varprojlim\sec_{\sigma_n} p_n,\\
&\sec^* p=\varprojlim\sec^* p_n,\\
&\sec^*_\sigma p=\varprojlim\sec^*_{\sigma_n} p_n.\\
\end{aligned}
$$
In particular, each component   $\sec_\sigma p$ or $\sec^*_\sigma p$ is not in general a nilpotent space but only pronilpotent. Nevertheless,
if $\pi_{\ge N}(F)=0$, classical obstruction theory let us build a homotopy inverse of $ \sec p\to \sec p_N$ and thus,
$$
\sec p\simeq\sec p_N.
$$
 In particular, each path component of $\sec p$ has the homotopy type of a nilpotent CW-complex of finite type. The same applies to the pointed situation.

This observation, together with all of the above for the case of a finite base, let us conclude in the following assertion:

If $F$ is a rational space with $\pi_{\ge N}(F)=0$, then $\sec p$ is a rational nilpotent space of the homotopy type of a CW-complex. Moreover,
$$
\sec p\simeq \sec p_N\simeq (\sec p_N)_\bq\simeq \sec {p_N}_\bq.
$$

Next, we recall specific Sullivan models of the spaces considered above. We will follow \cite{bufemu1,bumu1}, which is a slightly different approach of  \cite{bz}, based on the fundamental work of Haefliger \cite{haef}.

We first consider the base of
$
F\to E\stackrel{p}{\to}B,
$ to be a finite complex. Fix a model of this fibration
\begin{equation}\label{modelo}
 (A,d)\longrightarrow (A\otimes\Lambda V,D)\longrightarrow (\Lambda V,d),
 \end{equation}
in which $A$ is finite dimensional and denote by
$$(A^{\sharp},\delta)=(\Hom(A,\bq),d^\sharp)$$
 the differential graded coalgebra dual of $A$. Consider the free commutative $\bz$-graded algebra
   $$
   \Lambda( A\otimes\Lambda V\otimes A^\sharp),
   $$
   endowed with the differential induced by  $D$ and $\delta$. Let $J$ be the differential ideal generated by $1\otimes1-1$ and the elements
$$
\begin{aligned}
&v_1v_2\otimes\beta-\displaystyle \sum_j(-1)^{\abs{v_2}\abs{\beta'_j}}(v_1\otimes \beta'_j)(v_2\otimes\beta''_j)\text{,}\\
&b\otimes\alpha\otimes\beta-\displaystyle\sum_j(-1)^{\abs{\beta'_j}(\abs{ \alpha}+1)}\beta'_j(b)\alpha\otimes\beta''_j\text{,}
\end{aligned}
$$
with $v_1,v_2\in V$, $\alpha\in\Lambda V$, $b\in A$, $\beta\in A^\sharp$ and $\Delta\beta=\displaystyle\sum_j\beta'_j\otimes\beta''_j$.

The inclusion $V\hookrightarrow A\otimes\Lambda V$ induces an isomorphism of graded algebras
\begin{equation}\label{ecuation:definicionRho}
\rho\colon \Lambda(V\otimes A^{\sharp})\stackrel{\cong}{\longrightarrow}\Lambda(A\otimes\Lambda V\otimes A^{\sharp})/J\text{,}
\end{equation}
so that $\Lambda(V\otimes A^{\sharp})$ inherits a differential $\widetilde{d}$ which makes $\rho$ an isomorphism of differential graded algebras.  Then,

 %y entonces $\widetilde{d}$ define una diferencial en $\Lambda(V\otimes A^{\sharp})$ para la %cual $(\Lambda(V\otimes A^{\sharp}),\widetilde{d})$ es un modelo del espacio de secciones %$\sec(p)$.

\begin{theorem}\label{thm::modeloEspacioSecciones}
{\em \cite{bz,bufemu1}} $(\Lambda(V\otimes A^{\sharp}),\widetilde{d}\,)$ is a model of $\sec p$ and $(\Lambda(V\otimes A_+^{\sharp}),\widetilde{d}\,)$ is a model of $\sec^* p$.
\end{theorem}

In the case of a trivial fibration
 $Y\to X\times Y\to X$ choose a model
 $$A\to A\otimes(\Lambda V,d)\to (\Lambda V,d)$$
  in which $A$ is a finite dimensional  {\cdga}  model of $X$ and $(\Lambda V,d)$ is a Sullivan model of $Y$.  Then, it is enough to consider the differential ideal   $J$ of $\Lambda (\Lambda V\otimes A^\sharp)$ generated by
  $1\otimes 1^*-1$ and the elements of the form
$$v_1v_2\otimes \beta -\sum_j (-1)^{|v_2| |{\beta_j}' |}(v_1\otimes {\beta_j}'
)(v_2\otimes {\beta _j}'' ),$$
in which  $v_1,v_2\in V,\ \beta \in A^\sharp$ and $ \Delta\beta =\sum_j
{\beta_j}'\otimes {\beta _j}''$. Again,
$$\rho\colon   \Lambda (V\otimes A^\sharp)\subset   \Lambda (\Lambda V\otimes A^\sharp)\longrightarrow    \Lambda (\Lambda V\otimes
A^\sharp)/J$$
is an isomorphism af graded algebras which induces a differential $\widetilde d$ in $\Lambda (V\otimes A^\sharp)$  such that,

 \begin{theorem}{\em \cite{bz,bufemu1}}$(\Lambda (V\otimes A^\sharp),\widetilde d\,)$ (respectively $(\Lambda (V\otimes A^\sharp_+),\widetilde d\,)$ is a model of
$\map(X,Y)$ (respectively $\map^*(X,Y)$).
\end{theorem}

To deduce, from the setting above, models  of the path components of $\sec p$ or $\sec^* p$ we follow the approach in \cite[\S4]{bumu1}: given an augmentation $\varphi\colon (\Lambda W,d)\to \bq $ of a given  {\cdga}, consider $K_\varphi\subset \Lambda W$ the differential ideal generated by
$W^{< 0}$, $d W^0$ and $\{ \alpha -\varphi(\alpha ),\, \alpha \in W^0 \}$. Then,
$(\Lambda W, d )/K_\varphi$ is the  {\cdga}
$$(\Lambda (\overline{W}^1\oplus W^{\geq 2}),d_\varphi)$$
 in which $\overline{W}^1$ can be seen as the complement of $d(W^0)$ in $W^1$ modulo the
identifications in $K_\varphi$.

Now, if $(\Lambda W,d)$ is a model of $Z$ and $\varphi$ corresponds to a  $0$-simplex of $Z$, then $(\Lambda (\overline{W}^1\oplus W^{\geq 2}),d_\varphi)$ is a Sullivan model of the path component of $Z$ containing the given $0$-simplex  \cite[4.3]{bumu1}.

 In particular, choose $\sigma\colon B\rightarrow E$  a section of $p$ and $\varphi\colon{(A\otimes\Lambda V,D)}\to{(A,d)}$ a retraction of (\ref{modelo}) modeling  $\sigma$. Observe that $\varphi$ corresponds to an augmentation which we denote in the same way $\varphi\colon (\Lambda(V\otimes A^{\sharp}),\widetilde{d}\,)\to \bq$. Then,

\begin{theorem}\label{thm::modeloComponenteEspacioSecciones}
{\em \cite{bz,bufemu1,bumu1}}
The projection
$$(\Lambda (V\otimes {A^\sharp}),\widetilde{d}\,)\longrightarrow (\Lambda (V\otimes {A^\sharp}),\widetilde{d}\,)/K_\varphi\cong (\Lambda (\overline{V\otimes {A^\sharp}}^1\oplus (V\otimes
{A^\sharp})^{\geq 2}),d_\varphi)$$ is a model of the inclusion {\em$\sec_\sigma(p)\hookrightarrow\sec(p)$}.
In the same way,
$$(\Lambda (V\otimes {A^\sharp_+}),\widetilde{d}\,)\longrightarrow (\Lambda (V\otimes {A^\sharp_+}),\widetilde{d}\,)/K_\varphi\cong (\Lambda (\overline{V\otimes {A^\sharp_+}}^1\oplus (V\otimes
{A^\sharp_+})^{\geq 2}),d_\varphi)$$ is a model of the inclusion {\em$\sec_\sigma^*(p)\hookrightarrow\sec^*(p)$}.
\end{theorem}

We now make no finiteness assumptions on the base $B$ of $p$. Recall that, if $A$ is a  {\cdga}  model of the space $X$, then the inclusion  $X^{(n)}\subset X$ of the $n$-skeleton is modeled by the projection
$$
A\twoheadrightarrow A_n=A/I,
$$
where $I=A^{\ge n+1}\oplus C^n$ and $C^n$ is a complement of the cocycles of $A^n$.

Then, if (\ref{modelo}) denotes again a model of $p$, the sequence
$$
A_n\longrightarrow (A_n\otimes \Lambda V,D')\longrightarrow (\Lambda V,d),
$$
in which
$$
(A_n\otimes \Lambda V,D')= A_n\otimes_A(A\otimes \Lambda V,D),
$$
is a model of the pullback fibration
$$
F\longrightarrow E_n\stackrel{p_n}{\longrightarrow} B^{(n)}.
$$
Hence, at the sight of all of the above, we have:

\begin{theorem}\label{modelotruncado}
If $\pi_{\ge N}(F)=0$, then  $(\Lambda(V\otimes A_N^{\sharp}),\widetilde{d}\,)$ is a model of  $\sec p$ while  $(\Lambda(V\otimes {A_N}_+^{\sharp}),\widetilde{d}\,)$ models $\sec^* p$.\hfill$\square$
\end{theorem}
Finally, an analogous result to Theorem \ref{thm::modeloComponenteEspacioSecciones}, replacing $A$ by $A_N$, provides models for the  path components of  $\sec p$.

\section{Actions on elliptic spaces}

As  stated in the Introduction, $G$ will always denote a compact connected Lie group, or more generally a path connected topological group of the homotopy type of a finite CW-complex. Also, any considered topological $G$-space  $X$ will be of the homotopy type of a nilpotent CW-complex of finite type.

Given  a $G$-action $G\to\map(X,X)$, the rationalization functor   determines the homotopy type of a map $G\to\map(X_\bq,X_\bq)$. We see now that within this homotopy type there exist a representative which is in fact a $G$-action. This is well known if $X$ is a $G$-CW-complex (see for instance \cite[Chap. 1]{allpu} for the basics of these complexes). In fact \cite{may3}, \cite[II.3]{may},\cite[Thm. 10]{maymctrian}, $X$ always admits an {\em equivariant rationalization}
$\ell\colon X\longrightarrow  X_\bq$, that is, $X_\bq$ is also a  $G$-CW-complex, $\ell$ is an equivariant map and, for each closed subgroup $H$ of $G$, the map induced by $\ell$ on the fixed points $
 \ell^H\colon X^H\longrightarrow (X_\bq)^H,
 $
 is also a rationalization. In particular,
 $
 (X_\bq)^H\simeq (X^H)_\bq.
 $
In our slightly more general framework we proceed in a different way. Pulling back the rationalization $ X_\bq\to (X_{hG})_\bq\stackrel{\xi_\bq}{\to}  BG_\bq$ of the Borel fibration $
X\to X_{hG}\stackrel{\xi}{\to} BG$
 $$\xymatrix{
{X_\bq}\ar[d]\ar@{=}[r]&X_\bq \ar[d] \\
\pushoutcorner{E}\ar[r]\ar[d]&\left(X_{hG}\right)_\bq \ar[d]^{\xi_\bq} \\
{BG}\ar[r]&BG_\bq.
}
$$
we find the fibration $X_\bq\to E\to BG$ which sits in the following pullback diagram,
$$\xymatrix{
%&
G \ar@{=}[r] \ar[d] & G \ar[d] \\
%{X_\bq}\ar[r]\ar@{=}[d]&
X_\bq \ar[r] \ar[d] & EG \ar[d] \\
%{X_\bq}\ar[r]&
E \ar[r] & BG.
}
$$
As the right fibration is a principal $G$-bundle, $X_\bq$ inherits a $G$-action. We call this the {\em rational action on $X_\bq$} associated to the original $G$-action on $X$. Observe that the rationalization $X\to X_\bq$ is equivariant with respect to these actions and, in particular, $(X_\bq)^G\not=\emptyset$ as long as $X^G\not=\emptyset$. Moreover,

\begin{proposition}\label{prop::PostnikovPiece}
 If  $X$ is a Postnikov piece, that is $\pi_{>N}X=0$ for some $N$, then:
\begin{itemize}
\item[(i)] $X^{hG}$ has the homotopy type of a nilpotent CW-complex of finite type.
\item[(ii)]
$
(X^{hG})_\bq\simeq(X_\bq)^{hG}.
$
\end{itemize}
\end{proposition}

\begin{lemma} The rationalizations of the  Borel fibrations $\xi$ and $\eta$, of the original action and the associated rational action respectively, coincide. In particular $(X_{hG})_\bq\simeq (X_\bq)_{hG}$
 \end{lemma}

  \begin{proof} As the Borel construction is natural with respect to $G$-equivariant maps, $X\to X_\bq$ induces the diagram
  $$\xymatrix{
{X}\ar[r]\ar[d]&{X_\bq}\ar[d]\\
X_{hG}\ar[r]\ar[d]^{\xi}&{(X_\bq)_{hG}}\ar[d]^\eta\\
{BG}\ar@{=}[r]&{BG\text{.}}
}$$
As the top and bottom maps are  rational equivalence so is the middle one.
  \end{proof}

\begin{proof}[of Proposition \ref{prop::PostnikovPiece}] The basic observations in \S2 clearly imply (i) as\break $X^{hG}\simeq\sec\xi$.

For the second assertion, and with the notation in \S2, $\sec \xi\simeq \sec \xi_N$, and thus $(\sec \xi)_\bq\simeq (\sec \xi_N)_\bq$. Again, by the general statements in the past section, $(\sec\xi_N)_\bq\simeq \sec {\xi_N}_\bq$. Now, by the lemma above, $\sec {\xi_N}_\bq\simeq  {\sec \eta_N}_\bq$. But note that the fibre of $\eta_N$ is a rational Postnikov piece. Hence $\sec\eta_N$ is rational and ${\sec \eta_N}_\bq\simeq \sec \eta_N$. Finally, observe that $\sec\eta_N\simeq \sec \eta$. This chain of identities reduces to $
(\sec\xi)_\bq\simeq\sec\eta
$. That is, $
(X^{hG})_\bq\simeq(X_\bq)^{hG}.
$
\end{proof}

\begin{remark}\label{remarko} {\em If the $G$-space $X$ is not a Postnikov piece, the homotopy fixed point set $X^{hG}$
may not be of the homotoy type of a CW-complex and thus, $(X^{hG})_\bq$ makes no sense from the classical point of view. However, if $\pi_{>N} X$ is torsion for some $N$, in particular if $X$ is rationally elliptic, $X_\bq$ is a Postnikov piece and by (i) of Proposition \ref{prop::PostnikovPiece}, the homotopy fixed point set of the rational action $(X_\bq)^{hG}$ is a nilpotent space of the homotopy type of a CW-complex. Moreover, in view of the general facts in \S2, $(X_\bq)^{hG}$is also a rational space for which,
$$
 (X_\bq)^{hG}\cong \sec \eta\simeq\sec \eta_N\simeq \sec {\eta_N}_\bq,
$$
where, as above, $\eta$ is the
 Borel fibration of the associated rational $G$-action on $X_\bq$.
.}
\end{remark}

From now on we will consider $G$-actions on rationalizations $X_\bq$ of elliptic spaces (which are not necessarily arising from actions in $X$). In view of the remark above, as there is no possible confusion, we will denote $(X_\bq)^{hG}$ simply by $X_\bq^{hG}$.

Our main result is the following.

\begin{theorem}\label{principal1}
Let $X$ be an elliptic space for which $X_\bq$ is a $G$-space. Then, each path component of the homotopy fixed point set $X_\bq^{hG}$ is also elliptic.
\end{theorem}

The remaining of the section is devoted to the proof of the result above.

 Assume that $\pi_{\ge N}X_\bq=0$. Then, as noted in Remark \ref{remarko},  $X_\bq^{hG}$ is rational, nilpotent and
$$
 X_\bq^{hG}\cong \sec \xi\simeq\sec \xi_N\simeq \sec {\xi_N}_\bq,
$$
where
$
X_\bq\longrightarrow (X_\bq)_{hG}\stackrel{\xi}{\longrightarrow}BG
$
is the Borel fibration associated to the $G$-action on $X_\bq$.

The minimal model of $G$ is an exterior algebra $(\Lambda P,0)$ where $P$ is an oddly graded finite dimensional space \cite[Pág. 143, Ex.3]{fehatho}. Hence \cite[Prop. 15.15]{fehatho}, the minimal model of $BG$ is $(\Lambda Q,0)$ with $Q^r=P^{r-1}$, and any relative Sullivan algebra modeling $\xi_N$ is of the form
\begin{equation}\label{hacefalta}
A\longrightarrow (A\otimes\Lambda V,D)\longrightarrow (\Lambda V,d),
\end{equation}
where
$A=(\Lambda Q/(\Lambda Q)^{>N},0)$
and $(\Lambda V,d)$ is the minimal model of $X$. Hence, in view of Theorem  \ref{modelotruncado}, a model of $
 X_\bq^{hG}$ is
 $$
 \bigl(\Lambda (V\otimes A^\sharp),\widetilde d\,\bigr).
$$
To model the path component of  $
 X_\bq^{hG}$ containing a given section we can always assume, modifying the differential $D$, that this section is modeled by the retraction  $\varphi\colon(A\otimes\Lambda V,D)\rightarrow A$, $\varphi(V)=0$. In particular, $DV\subset A\otimes\Lambda^+ V$.
 In this case, Theorem \ref{thm::modeloComponenteEspacioSecciones} asserts that a Sullivan model of the considered component is,
$$
\bigl(\Lambda (V\otimes {A^\sharp}),\widetilde{d}\,\bigr)/K_\varphi\cong \bigl(\Lambda (\overline{V\otimes {A^\sharp}}^1\oplus (V\otimes
{A^\sharp})^{\geq 2}),\widetilde d\,\bigr).
$$
Denote by
\begin{equation}\label{puroseccion} \bigl(\Lambda (\overline{V\otimes {A^\sharp}}^1\oplus (V\otimes
{A^\sharp})^{\geq 2}),\widetilde d_\sigma\bigr)
\end{equation}
its associated pure model.

Now, ellipticity of $X$ translates, via Theorem \ref{prop:puroEliptico}, to  the existence of a homogeneous ordered basis of
 $V$,
 $$
 \base_V=\{x_1,\dots,x_r,y_1,\dots,y_s\},
 $$
 where $\{x_1,\dots,x_r\}$, $\{y_1,\dots,y_s\}$ are basis of $V^\text{even}$ and $V^\text{odd}$ respectively, such that: for any $i=1,\dots,r$, there is a positive integer $N_i$, and elements
 $$\Psi_i\in\Lambda V,\ \ \Theta_i\in\Lambda^+(x_1,\dots,x_{i-1})\cdot\Lambda V,$$
  for which
\begin{equation}\label{formulaperita}
d_\sigma \Psi_i=x_i^{N_i}+\Theta_i.
\end{equation}
We will use the same criterium to show the elliptic character of (\ref{puroseccion}).
To do so, given  $\{a_1,\dots,a_t\}$
 a homogeneous basis   of $Q$, consider the induced basis of  $A$,
$$
\mathfrak{B}_A=\{a_1^{j_1}a_2^{j_2}\cdots a_t^{j_t}\},\quad j_i\ge 0, \quad i=1,\dots t,\quad \abs{a_1^{j_1}a_2^{j_2}\cdots a_t^{j_t}}\leq N,
$$
and the corresponding dual basis in  $A^\sharp$,
$$
\mathfrak{B}_{A^\sharp}=\{(a_1^{j_1}a_2^{j_2}\cdots a_t^{j_t})^\sharp\},\quad j_i\ge 0, \quad i=1,\dots t,\quad \abs{a_1^{j_1}a_2^{j_2}\cdots a_t^{j_t}}\leq N.
$$
The computation of the diagonal for the element of this basis
is straightforward:
\begin{equation}\label{delta}
 \Delta(a_1^{j_1}\cdots a_t^{j_t})^{\sharp}=\sum p^{\sharp}\otimes q^{\sharp},\quad  p,q\in \base_A,\,\,
 pq=a_1^{j_1}\cdots a_t^{j_t}.
 \end{equation}
Moreover, if we denote $\Delta_2=\Delta$ and $\Delta_n=(1_{A^\sharp}\otimes\dots\otimes 1_{A^\sharp}\otimes\Delta_2)\circ \Delta_{n-1}$, then,
\begin{equation}\label{deltan}
\Delta_n(a_1^{j_1}\cdots a_t^{j_t})^{\sharp}=\sum p_{i_1}^{\sharp}\otimes\cdots\otimes p_{i_n}^{{\sharp}},\quad p_{i_j}\in\base_A, \,\, p_{i_1}\dots p_{i_n}=a_1^{j_1}\cdots a_t^{j_t}.
\end{equation}

Now, we consider the basis of $\overline{V\otimes {A^\sharp}}^1\oplus (V\otimes
{A^\sharp})^{\geq 2}$ given by
$$
\mathfrak{B}=\{
v\otimes \alpha\},\quad v\in\mathfrak{B}_V,\quad \alpha\in\mathfrak{B}_{A^{\sharp}},\quad v\otimes\alpha\in \overline{V\otimes {A^\sharp}}^1\oplus (V\otimes
{A^\sharp})^{\geq 2}
\text{,}$$
 and define the following order on its elements.

 Given  $x_k,x_\ell\in\base_V$, we say that $x_k\otimes 1$ {\em precedes}  $x_\ell\otimes 1$, and we denote it by
$
x_k\otimes 1\prec x_\ell\otimes 1$, if $k<\ell$.

In general,
 given even elements
$$
x_k\otimes(a_1^{j_1}a_2^{j_2}\cdots a_t^{j_t})^{\sharp},\,\,x_\ell\otimes(a_1^{j_1'}a_2^{j_2'}\cdots a_t^{j_t'})^{\sharp}\in\base,
$$
where some $j_r$ and $j'_s$ are non zero,  $x_k\otimes(a_1^{j_1}a_2^{j_2}\cdots a_t^{j_t})^{\sharp}$ \emph{precedes}  $x_\ell\otimes(a_1^{j_1'}a_2^{j_2'}\cdots a_t^{j_t'})^{\sharp}$, and we denote it by
$$
x_k\otimes(a_1^{j_1}a_2^{j_2}\cdots a_t^{j_t})^{\sharp} \prec x_\ell\otimes(a_1^{j_1'}a_2^{j_2'}\cdots a_t^{j_t'})^{\sharp},
$$
whenever:

\begin{itemize}
\item[1.] $\displaystyle \frac{\abs{x_k}}{\abs{(a_1^{j_1}a_2^{j_2}\cdots a_t^{j_t})^{\sharp}}} <\frac{\abs{x_\ell}}{\abs{(a_1^{j_1'}a_2^{j_2'}\cdots a_t^{j_t'})^{\sharp}}}$; or else,

\item[2.]
$\displaystyle\frac{\abs{x_k}}{\abs{(a_1^{j_1}a_2^{j_2}\cdots a_t^{j_t})^{\sharp}}} =\frac{\abs{x_\ell}}{\abs{(a_1^{j_1'}a_2^{j_2'}\cdots a_t^{j_t'})^{\sharp}}}$ and $k<\ell$; or else,

\item[3.]
$\displaystyle\frac{\abs{x_k}}{\abs{(a_1^{j_1}a_2^{j_2}\cdots a_t^{j_t})^{\sharp}}} =\frac{\abs{x_\ell}}{\abs{(a_1^{j_1'}a_2^{j_2'}\cdots a_t^{j_t'})^{\sharp}}}$, $ x_k=x_\ell$, and
$$
(a_1^{j_1}a_2^{j_2}\cdots a_t^{j_t})^{\sharp}<_{A^\sharp}(a_1^{j_1'}a_2^{j_2'}\cdots a_t^{j_t'})^{\sharp}.
$$
\end{itemize}
Here,  $<_{A^\sharp}$ denotes the {\em lexicographic order} in  $\mathfrak{B}_{A^{\sharp}}$, that is,
$$
(a_1^{j_1}a_2^{j_2}\cdots a_t^{j_t})^{\sharp}<_{A^\sharp} (a_1^{j'_1}a_2^{j'_2}\cdots a_t^{j'_t})^{\sharp},
$$
if there is $i=1,\dots,t$ such that $j_i<j'_i$ and $j_l=j'_l$ for $l<i$.

Finally we choose any order on the odd elements of $\base$ and impose that  even elements always precede odd elements.

 We will devote the rest of the proof to show that, with this order, Theorem \ref{prop:puroEliptico}(iii) holds. Choose an even element
$x_i\otimes(a_1^{j_1}a_2^{j_2}\cdots a_t^{j_t})^{\sharp}\in\mathfrak{B}$ and fix, for $x_i$, an integer
 $N_i$ and elements
$$\Psi_i\in\Lambda V,\quad\Theta_i\in\Lambda^+(x_1,\dots,x_{i-1})\cdot\Lambda V,
$$
 satisfying equation (\ref{formulaperita}).
Consider,
 $$
 \begin{aligned}
 \alpha &=(a_1^{j_1N_i}a_2^{j_2N_i}\cdots a_t^{j_tN_i})^{\sharp}\in\base_{A^\sharp},\\
 \eta_i &=\rho^{-1}[\Psi_i\otimes\alpha]\in  \Lambda (\overline{V\otimes {A^\sharp}}^1\oplus (V\otimes
{A^\sharp})^{\geq 2}).
\end{aligned}
$$
Then, the proof of Theorem \ref{principal1} is completed with the following lemma, with $\rho$ as in (\ref{ecuation:definicionRho}).

\begin{lemma} With the notation above,
$$\widetilde{d}_\sigma\eta_i=\bigl(x_i\otimes(a_1^{j_1}a_2^{j_2}\cdots a_t^{j_t})^{\sharp} \bigr)^{N_i} + \nu_i,
$$
with  $\nu_i$ in the ideal $I$ generated by
$\{v\otimes\beta\in\mathfrak{B},\,\, v\otimes\beta\prec x_i\otimes(a_1^{j_1}a_2^{j_2}\cdots a_t^{j_t})^{{\sharp}} \}\text{.}$
\end{lemma}

\begin{proof}
Note that, as the differential in  $A^\sharp$ vanishes,
$$
\widetilde d\eta_i=\widetilde d \rho^{-1}[\Psi_i\otimes\alpha]=\rho^{-1}[D\Psi_i\otimes \alpha].
$$
modulo the ideal $K_\varphi$.
Hence, if we write
$
 D\Psi_i=d_\sigma \Psi_i+ \Phi+\Gamma+\Omega$ with
 $$
 \Phi\in\Lambda^+V^{\text{odd}}\cdot \Lambda V,\quad \Gamma\in A^+\otimes\Lambda^+V^{\text{even}}, \quad \Omega\in A^+\otimes\Lambda^+V^{\text{odd}}\cdot\Lambda V,
 $$
we conclude that
\begin{eqnarray}
\nonumber  \widetilde{d}_\sigma\eta_i &=& \rho^{-1}[d_\sigma\Psi_i\otimes\alpha]+\rho^{-1}[\Gamma\otimes\alpha] \\
\nonumber   &=& \rho^{-1}[x_i^{N_i}\otimes\alpha]+\rho^{-1}[\Theta_i\otimes\alpha]+\rho^{-1}[\Gamma\otimes\alpha]\text{.} \end{eqnarray}
We analyze each summand separately.

Modulo scalars, write
$\Gamma =\sum a_1^{k_1}\cdots a_t^{k_t} \otimes x_{\ell_1}\cdots  x_{\ell_p}$ so that
 $$
\rho^{-1}[\Gamma\otimes\alpha]=\sum \rho^{-1}[(a_1^{k_1}\cdots a_t^{k_t}\otimes x_{\ell_1}\cdots x_{\ell_p})\otimes\alpha].
$$
Now, in view of (\ref{delta}),(\ref{deltan}), and the generators  of the ideal  $J$ (see \S2),  any summand of the above formula either vanishes (if $k_m>j_mN_i$ for some $m=1,\dots,t$), or else,
\begin{equation}\label{terminando1}
\begin{aligned}
&\rho^{-1}[(a_1^{k_1}\cdots a_t^{k_t}\otimes x_{\ell_1}\cdots x_{\ell_p})\otimes\alpha]=\\
&\rho^{-1}[x_{\ell_1}\cdots x_{\ell_p}\otimes(a_1^{j_1N_i-k_1}a_2^{j_2N_i-k_2}\cdots a_t^{j_tN_i-k_t})^{\sharp})]=\\
&\sum(x_{\ell_1}\otimes b_1^{\sharp})\cdots(x_{\ell_p}\otimes b_p^{\sharp})\text{,}
\end{aligned}
\end{equation}
where the summation is taken over all possible families  $\{b_1,\dots,b_p\}$ of the basis $\base_A$ such that,
$
b_1\cdots b_p=a_1^{j_1N_i-k_1}\cdots a_t^{j_tN_i-k_t}.
$

We now apply the trivial fact by which $\displaystyle\frac{r-k}{k-s}<-\frac{r}{s}$ for every  positive integers  $r>s>k$.  Choosing,
$$
a=|x_i^{N_i}|,\qquad b=-\abs{(a_1^{j_1N_i}a_2^{j_2N_i}\cdots a_t^{j_tN_i})^{{\sharp}}},\qquad k=|a_1^{k_1}\cdots a_t^{k_t}|,
$$
we have
$$\frac{\abs{x_i}}{\abs{(a_1^{j_1}a_2^{j_2}\cdots a_t^{j_t})^{{\sharp}}}}=\frac{\abs{x_i^{N_i}}}{\abs{(a_1^{j_1N_i}a_2^{j_2N_i}\cdots a_t^{j_tN_i})^{{\sharp}}}}>\frac{\abs{x_{\ell_1}\cdots x_{\ell_p}}}{\abs{(a_1^{j_1N_i-k_1}a_2^{j_2N_i-k_2}\cdots a_t^{j_tN_i-k_t})^{\sharp}}}\text{.}$$

 Next, for each summand in (\ref{terminando1}) we will apply another trivial fact: let  $r_1,r_2,\dots,r_n$ and $s_1,s_2,\dots,s_n$ be positive and negative integers respectively. Then, either:
\begin{equation}\label{trivialidad}
 \displaystyle\frac{r_i}{s_i}<\frac{\sum_jr_j}{\sum_js_j}\,\,\,\,\text{for some $i$, or else,}\,\,\,\,
\frac{r_i}{s_i}=\frac{\sum_jr_j}{\sum_js_j}\,\,\,\,\text{for all $i$.}
\end{equation}
Choosing $r_i=|x_{\ell_i}|$ and $s_i=|b_i^\sharp|$, we see that there exists  $m=1,\dots,p$ such that
 $$
 \frac{\abs{x_{\ell_1}\cdots  x_{\ell_p}}}{\abs{(a_1^{j_1N_i-k_1}a_2^{j_2N_i-k_2}\cdots a_t^{j_tN_i-k_t})^{\sharp}}}\geq\frac{\abs{x_{\ell_m}}}{\abs{b_m^{\sharp}}}\text{.}
 $$
 Hence,
 $$
 \frac{\abs{x_i}}{\abs{(a_1^{j_1}a_2^{j_2}\cdots a_t^{j_t})^{{\sharp}}}}> \frac{\abs{x_{\ell_m}}}{\abs{b_m^{\sharp}}},
 $$
 which translates to
 $$
 x_{\ell_m}\otimes b_m^{\sharp}\prec x_i\otimes(a_1^{j_1}a_2^{j_2}\cdots a_t^{j_t})^{\sharp}.
 $$
 In view of (\ref{terminando1}) this implies that
 \begin{equation}\label{uno}
  \rho^{-1}[\Gamma\otimes \alpha]\in I.
\end{equation}

Next, we focus on $\rho^{-1}[\Theta_i\otimes \alpha]$. Modulo scalars, write
$
\Theta_i=\sum \, x_{\ell_1}\cdots  x_{\ell_p}$, with $\ell_1<i$, so that,
$$
\rho^{-1}[x_{\ell_1}\cdots x_{\ell_p}\otimes\alpha]=\sum (x_{\ell_1}\otimes b_1^{\sharp})\cdots(x_{\ell_p}\otimes b_{p}^{\sharp})\text{,}
$$
where the summation is taken over families $\{b_1,\dots,b_p\}$ of $\base_A$ such that
$
b_1\cdots b_p=a_1^{j_1N_i}\cdots a_t^{j_tN_i}.
$
For each summand, applying again (\ref{trivialidad}) choosing $r_i=|x_{\ell_i}|$ and $s_i=|b_i^\sharp|$, we obtain that either:

 There exists some $m=1,\dots,p$ such that
 $$
 \frac{\abs{x_i}}{\abs{(a_1^{j_1}a_2^{j_2}\cdots a_t^{j_t})^{{\sharp}}}}=\frac{\abs{x_i^{N_i}}}{\abs{(a_1^{j_1N_i}a_2^{j_2N_i}\cdots a_t^{j_tN_i})^{{\sharp}}}}>\frac{\abs{x_{\ell_{m}}}}{\abs{b_m^{\sharp}}},
 $$
 in which case
 $
 x_{\ell_{m}}\otimes b_m^{\sharp}\prec x_i\otimes(a_1^{j_1}a_2^{j_2}\cdots a_t^{j_t})^{\sharp};
 $
or else,
for every $m=1,\dots,p$,
 $$
 \frac{\abs{x_i}}{\abs{(a_1^{j_1}a_2^{j_2}\cdots a_t^{j_t})^{{\sharp}}}}=\frac{\abs{x_i^{N_i}}}{\abs{(a_1^{j_1N_i}a_2^{j_2N_i}\cdots a_t^{j_tN_i})^{{\sharp}}}}=\frac{\abs{x_{\ell_{m}}}}{\abs{b_m^{\sharp}}}.
 $$
 However, as $\ell_1<i$, we obtain that
 $
 x_{\ell_{1}}\otimes b_1^{\sharp}\prec x_i\otimes(a_1^{j_1}a_2^{j_2}\cdots a_t^{j_t})^{\sharp}.
 $
Thus, we conclude that
 \begin{equation}\label{dos}
 \rho^{-1}[\Theta_i\otimes\alpha]\in  I.
 \end{equation}

Finally, we analyze
$$\rho^{-1}[x_i^{N_i}\otimes\alpha]=\rho^{-1}[x_i^{N_i}\otimes(a_1^{j_1N_i}\cdots a_t^{j_tN_i})^{\sharp}]=\sum (x_i\otimes b_1^{\sharp})\cdots(x_i\otimes b_{N_i}^{\sharp})\text{,}$$
where the summation is taken over families  $\{b_1,\dots,b_{N_i}\}$ of $\base_A$ such that
$
b_1\cdots b_{N_i}=a_1^{j_1N_i}\cdots a_t^{j_tN_i}.
$
For each summand, applying again (\ref{trivialidad}) above, we assert that either:

There  exists some  $m=1,\dots,N_i$ such that
 $$
 \frac{\abs{x_i}}{\abs{(a_1^{j_1}a_2^{j_2}\cdots a_t^{j_t})^{{\sharp}}}}=\frac{\abs{x_i^{N_i}}}{\abs{(a_1^{j_1N_i}a_2^{j_2N_i}\cdots a_t^{j_tN_i})^{{\sharp}}}}>\frac{\abs{x_i}}{\abs{b_m^{\sharp}}},
 $$
  in which case
  $
  x_i\otimes b_m^{\sharp}\prec x_i\otimes(a_1^{j_1}a_2^{j_2}\cdots a_t^{j_t})^{\sharp};
  $
or else,
  for every $m=1,\dots,N_i$,
  $$
  \frac{\abs{x_i}}{\abs{(a_1^{j_1}a_2^{j_2}\cdots a_t^{j_t})^{{\sharp}}}}=\frac{\abs{x_i^{N_i}}}{\abs{(a_1^{j_1N_i}a_2^{j_2N_i}\cdots a_t^{j_tN_i})^{{\sharp}}}}=\frac{\abs{x_i}}{\abs{b_m^{\sharp}}}\text{.}
  $$
 In this case, as $
b_1\cdots b_{N_i}=a_1^{j_1N_i}\cdots a_t^{j_tN_i},
$
the inequality $(a_1^{j_1}\cdots a_t^{j_t})^{{\sharp}}<_{A^{\sharp}}b_m$ cannot hold for every $m=1,\dots,N_i$.

Hence there are two possibilities: one is that $b_m^{\sharp}<_{A^{\sharp}}(a_1^{j_1}\cdots a_t^{j_t})^{{\sharp}}$ for some  $m=1,\dots,N_i$ in which case
$
x_i\otimes b_m^{\sharp}\prec x_i\otimes(a_1^{j_1}a_2^{j_2}\cdots a_t^{j_t})^{\sharp}.
$ The other is that, for every $m=1,\dots, N_i$, $b_m^\sharp=(a_1^{j_1}\cdots a_t^{j_t})^{{\sharp}}$. This only occurs when the considered summand is
$
(x_i\otimes (a_1^{j_1}a_2^{j_2}\cdots a_t^{j_t})^{\sharp})^{N_i}.
$
Hence, we conclude that
\begin{equation}\label{tres}
\rho^{-1}[x_i^{N_i}\otimes\alpha]=(x_i\otimes (a_1^{j_1}a_2^{j_2}\cdots a_t^{j_t})^{\sharp})^{N_i}+\Upsilon
\end{equation}
where
$
\Upsilon\in I\{ v\otimes\beta\in\mathfrak{B},\quad v\otimes\beta\prec x_i\otimes(a_1^{j_1}a_2^{j_2}\cdots a_t^{j_t})^{{\sharp}}\}\text{.}
$

From (\ref{uno}), (\ref{dos}) and (\ref{tres}) the lemma follows and Theorem \ref{principal1} is proved.
\end{proof}

\section{Tori and circle actions}
In this section we present some consequences and examples of the results in the past section when $G$ is a torus. We begin by a consequence of Theorem \ref{principal1} for circle actions.

 \begin{corollary}\label{corolario1}
 Let $X$ be a nilpotent finite complex in which $\bs^1$ acts. Then, $X$ is elliptic if and only if each component of the homotopy fixed point set $X_\bq^{h\bs^1}$ has finite Lusternik-Schnirelmann category.
 \end{corollary}

 \begin{proof}
By Theorem \ref{principal1}, if $X$ is ellptic, then each component of $X_\bq^{h\bs^1}$ is also elliptic and in particular has finite LS-category. On the other hand, if $X$ is hyperbollic, Theorem 5 of \cite{bufemu2} guarantees that the LS-category of such a component is infinite. Note that in this case $X_\bq^{h\bs^1}$ may not be a CW-complex, nor nilpotent.
\end{proof}

We now describe explicitly the homotopy type of the homotopy fixed point set of certain tori and circle actions. In the remaining of the section any considered action has fixed points.

\begin{theorem}\label{accionenesfera}
Given an $\bs^1$ action  in the rational $n$-sphere $\bs^n_{\bq}$:
\begin{itemize}
\item[(i)] If $n$ is odd,
 $$
 {\bs_\bq^n}^{h\bs^1}\simeq \map\left({\bc P}^{\frac{n+1}{2}},\bs^n\right)_\bq \simeq_\bq \bs^1\times\bs^3\times\cdots\times\bs^n.
 $$
\item[(ii)] If $n$ is even,  ${\bs^n_\bq}^{h\bs^1}$ is   either path component and of the rational homotopy type of
    $$
    \bs^1\times\frac{SO(n+2)}{U\left(\frac{n+2}{2}\right)},
$$
or else, it has two components, each of them of the rational homotopy type of
    $$
    \bs^{n+1}\times\bs^{n+3}\times\cdots\times\bs^{2n-1}.
     $$
\end{itemize}
\end{theorem}

\begin{proof} (i) In view of Remark \ref{remarko},
$$
{\bs^n_\bq}^{h\bs^1}\cong \sec\xi \simeq\sec \xi_{n+1}\simeq \sec{\xi_{n+1}}_\bq
 $$
 where
 $$
\bs^n_\bq\longrightarrow (\bs^n_\bq)_{h\bs^1}\stackrel{\xi}{\longrightarrow} {\bc P}^\infty
$$
is the corresponding Borel fibration. Note that $\xi_{n+1}$ is modeled by the trivial algebraic fibration
$$
(A,0)\longrightarrow (A\otimes\Lambda x,0)\longrightarrow (\Lambda x,0)
$$
where $A=(\Lambda a)/a^{\frac{n+1}{2}}$, $|a|=2$ and $|x|=n$. Thus,
$$
{\bs^n_\bq}^{h\bs^1}\simeq\sec \xi_{n+1}\simeq \mbox{map}\,(\mathbb CP^{\frac{n+1}{2}}, {\bs}^n)_{\mathbb Q}\,.$$
 Now since $n$ is odd, ${\bs}^n_{\mathbb Q}$ is an H-space, and  so is ${\bs}_{\mathbb Q}^{n\, h{\bs}^1}$. On the other hand,   a direct computation  shows that the minimal model for the mapping space is
 $$\left( \Lambda (x\otimes 1, x\otimes a^{\sharp}, \cdots , x\otimes (a^{\frac{n-1}{2}})^{\sharp}),0\right)\,.$$
 We observe that this is also the minimal model for the product ${\bs}^1\times {\bs}^3\times \cdots \times {\bs}^n$, and this gives the result.

 Note that a basis for the rational homotopy groups of ${\bs}_{\mathbb Q}^{n\, h{\bs}^1}$ is given by the maps
 $$\varphi_q \colon {\bs}^{n-2q}\to {\bs}_{\mathbb Q}^{n\, h{\bs}^1},\ \ q= 0, \cdots , \frac{n-1}{2},$$ whose adjoints are
given by  the following compositions:
 $${\bs}^{n-2q}\times \mathbb CP^{\frac{n+1}{2}} \to {\bs}^{n-2q}\times \mathbb CP^{\frac{n+1}{2}} \, / (*\times \mathbb CP^{\frac{n+1}{2}}) = \vee_{r=0}^{\frac{n+1}{2}} {\bs}^{n-2q+2r} \stackrel{p_q}{\to} {\bs}^n \,.$$
Here $p_q$ denotes the projection on the component $r=q$.
\smallskip

(ii)  Again by Remark \ref{remarko}, as $\pi_{\ge 2n}(\bs^n_\bq)=0$,
$$
{\bs^n_\bq}^{h\bs^1}\cong\Sec\xi \simeq\Sec \xi_{2n}\simeq \Sec{\xi_{2n}}_\bq.
 $$
Here,
$\xi_{2n}$ is modeled by
$$
(A,0)\longrightarrow (A\otimes\Lambda (x,y),D)\longrightarrow (\Lambda (x,y),d)
$$
where $
A=(\Lambda a)/a^{n+1}$, $a,x,y$ are of degree $2,n,2n-1$ respectively, , $Dx=0$ and $\displaystyle Dy=x^2+\lambda a^{\frac{n}{2}}x$,  $\lambda\in\bq$.

\smallskip

(ii.1) If $\lambda=0$ then $\xi_{2n}$ is trivial and
$$
{\bs^n_\bq}^{h\bs^1}\simeq \mbox{map}\,(\mathbb CP^{n}, {\bs}^n)_{\mathbb Q}\,.$$
A straightforward computation shows that this mapping space has a model of the form
$$(\Lambda (\, (x_s)_{1\leq s\leq n/2}\,, \, (y_r)_{1\leq r\leq n}\, ),d)$$
with $\vert x_s\vert = 2s$, $\vert y_r\vert = 2r-1$, $dy_1=0$, $dx_s=0$ and for $r>1$, $dy_r = \sum_{s+t=r} x_sx_t$.
This is the model of ${\bs}^1\times SO(n+2)/U(\frac{n+2}{2})$ \cite[Thm. 2.iii.5]{mu}.

\smallskip

(ii.2) If $\lambda\not=0$, the truncated Borel fibration $\xi_{2n}$ has two non homotopic sections $\sigma,\tau$ which corresponds to the only two possible retractions of its model,
$$
\varphi_\sigma,\varphi_{\tau}\colon  (A\otimes\Lambda (x,y),D)\to (A,0),\qquad \varphi_\sigma(x)=0,\quad \varphi_\tau(x)=-\lambda.
$$
A direct computation shows that the model of $\sec_\sigma \xi_{2n}$ is of the form,
$$
\bigl(\Lambda(x,x\otimes,a^\sharp,\dots, x\otimes (a^{n/2})^\sharp, y,y\otimes,a^\sharp,\dots, x\otimes (a^{n-1})^\sharp),\widetilde d\,\bigr)
$$
where the linear part of the differential satisfies,
$$\widetilde d_1(y\otimes (a^{n/2})^{\sharp}) = \lambda x\,, \cdots , \widetilde d_1(y\otimes (a^{n-1})^{\sharp}) = \lambda x\otimes (a^{N/2 - 1})^{\sharp}\,.$$
Therefore, if we quotient this model by the acyclic ideal generated by the variables $(x\otimes (a^r)^{\sharp})$ and by $(y\otimes (a^s)^{\sharp})$ for $n/2\leq s\leq n-1$, we obtain a model with differential zero
 $$(\Lambda (y, y\otimes a^{\sharp}, \cdots , y\otimes  (a^{n/2-1})^{\sharp}), 0)\,.$$ This implies the result.

 The same conclusion is obtained for the model of $\sec_\tau \xi_{2n}$ replacing $\lambda$ by $-\lambda$.

\end{proof}

\begin{theorem}\label{torito}
The Borel fibration of any torus $\bt^m$  action in an odd dimensional rational sphere $\bs_\bq^n$ is homotopically trivial. In particular,
$$
{\bs_\bq^n}^{h\bt^m}
\simeq
\map\bigl(({\bc P}^\infty\times\stackrel{m}{\cdots}\times{\bc P}^\infty)^{(n+1)} ,\bs^n_\bq\bigr)
,$$ which has the rational homotopy type of a finite product of odd dimensional spheres.
\end{theorem}

\begin{proof}
This is a generalization of Theorem \ref{accionenesfera} (i). Indeed,  a direct computation shows that a model for the homotopy fixed point set is given by
$\bigl( \Lambda (x\otimes J),0\bigr)$
with $x$ in degree $n$ and
$$J=\{(a_1^{i_1}a_2^{i_2}\cdots a_n^{i_m})^{\sharp}, \,\, \sum i_s \leq \frac{n-1}{2}\,\}.$$ This gives the result.
\end{proof}

\begin{theorem}\label{complejito}
Given an  $\bs^1$-action  in the complex projective space  $\bc P^n_\bq$, the homotopy fixed point set  ${\bc P^n_\bq}^{\,h\bs^1}$ has at most $n+1$ path components and each of them is of the rational homotopy type of one of the following spaces:
$$
\begin{aligned}
&\quad\bs^3\times\bs^5\times\dots\times\bs^{2n+1},\\
\bs^{1}\times&{\bc P}^1\times\bs^{5}\times\dots\times\bs^{2n+1},\\
\bs^{1}\times&\bs^3\times{\bc P}^2\times\dots\times\bs^{2n+1},\\
\dots&\\
\bs^{1}\times& \bs^3\times\dots\times {\bc P}^{n-1}\times\bs^{2n+1},\\
\bs^{1}\times &\bs^3\times\dots\times \bs^{2n-1}\times{\bc P}^n.\\
\end{aligned}
$$
\end{theorem}

\begin{proof} As $\pi_{\ge 2n+2}(\bc P^n_\bq)=0$,
$$
{\bc P^n_\bq}^{\,h\bs^1}\simeq \sec \xi\simeq \sec \xi_{2n+2}\simeq\sec{\xi_{2n+2}}_\bq.
$$
where
$$
\bc P^n_\bq\longrightarrow (\bc P^n_\bq)_{h\bs^1}\stackrel{\xi}{\longrightarrow}\bc P^\infty,
$$
is the corresponding Borel fibration. A model of $\xi_{2n+2}$ is
$$
(A,0)\longrightarrow(A\otimes \Lambda(x,y),D)\longrightarrow (\Lambda (x,y),d),
$$
where $
A=(\Lambda a)/a^{n+2}$, $|a|=|x|=2$, $|y|=2n+1$, and
 $$
 Dx=0,\qquad Dy=x^{n+1}+\sum_{j=1}^{n}\lambda_ja^jx^{n+1-j},\quad\lambda_j\in\bq,\,\, j=1,\dots,n.
 $$
 Non homotopic sections of $\xi_{2n+2}$ correspond to non homotopic retractions
  $$
  (A\otimes \Lambda(x,y),D)\longrightarrow (A,0)
  $$
  each of which is uniquely determined by the scalar $\mu\in\bq$ for which $\varphi(x)=\mu a$. However,
  the equation $\varphi(Dy)=0$ reduces the number of these scalars  to the set of rational solutions of the equation
  $$
  z^{n+1}\sum_{j=1}^{n}\lambda_jz^{n+1-j}=0.
  $$
This proves that ${\bc P^n_\bq}^{\,h\bs^1}$ has at most $n+1$ components.

On the other hand, a direct computation shows that any the resulting path component of $\sec \xi_{2n+2}$ has a model of the form
 $$
 \bigl(\Lambda (x,y,y\otimes a^\sharp,\dots,y\otimes (a^n)^\sharp,d\bigr),\qquad d y\otimes (a^i)^\sharp=\mu_i x^i,\quad i=1,\dots,n,
$$
for some scalars $\mu_i\in \bq$, $i=1,\dots,n$.
\smallskip

 If $\mu_1\not=0$ this is a model of
$$
\bs^{3}\times\bs^{5}\times\cdots\times\bs^{2n+1}\text{.}
$$
\smallskip

More generally, if $\mu_1=\dots=\mu_{i-1}=0$ and $\mu_i\not=0$ this is a model of
$$
\bs^{1}\times\dots\times \bs^{2i-3}\times{\bc P}^n\times \bs^{2i+1}\times\cdots\times\bs^{2n+1}\text{.}
$$
\smallskip

Finally, if $\lambda_i=0$ for all $i=1,\dots,n$, the above is a model of

$$
\bs^{1}\times\bs^3\times\bs^5\times\cdots\times\bs^{2n-1}\times{\bc P}^n\text{.}
$$

\end{proof}
\section{The inclusion $
k\colon X^G\hookrightarrow X^{hG}
$}

Let $X$ be a $G$-space for which $X^G\not=\phi$.  The equivariant map $
X^G\hookrightarrow X
 $ induces a map between the corresponding Borel constructions,
\begin{equation}\label{diagramaprin}
\xymatrix{
X^G\ar@{^(->}[rr] \ar[d]&&X\ar[d]\\
BG\times X^G\ar[rd]_\eta\ar[rr]^(.52)\gamma&&X_{hG}\ar[ld]^\xi\\
&BG.&
}
\end{equation}
Then, the fundamental inclusion from the fixed point set into the homotopy fixed point is identified with
$$
 k\colon X^G\hookrightarrow \map(BG,X^G)\stackrel{\gamma_*}{\longrightarrow}\sec \xi\cong X^{hG}.
$$
 where $X^G\hookrightarrow \map(BG,X^G)$ and $\gamma_*\colon \sec \eta\to \sec \xi$ are induced by $BG\to *$ and $\gamma$ respectively.

From now on, to avoid excessive notation, the $G$-space $X$  will  be already a rational nilpotent space with $\dim \pi_*(X)<\infty$, in particular any elliptic space. Moreover, $X^G$ will be assumed to be of the homotopy type of a path connected (otherwise, everything which follows can be applied to any of its path components) nilpotent CW-complex of finite type (over $\bq$) and with rational homotopy groups of total finite dimension.  In some cases, for tori for instance as we shall see later, the fact that  $\dim \pi_*(X^G)\otimes\bq<\infty$ is deduced from $\dim \pi_*(X)<\infty$. In general however, not much is known from the homotopical behaviour of $X^G$ unless some constrains are imposed (see for instance  \cite[\S3.3]{allpu} or \cite{changskjel}).

 Under these conditions, in view of Proposition \ref{prop::PostnikovPiece} and Remark \ref{remarko}, each component of   $\map(BG,X^G)$ and $X^{hG}$ has the homotopy type of a rational nilpotent CW-complex. Moreover, if $N$ is such that $\pi_{\ge N}(X)=\pi_{\ge N}(X^G)\otimes\bq=0$, the map $k$ above is identified with the corresponding
  \begin{equation}\label{ktruncada}
X^G\hookrightarrow \map(BG^{(N)},X^G)\stackrel{{\gamma_N}_*}{\longrightarrow}\sec \xi_{N}\cong X^{hG},
 \end{equation}
 obtained by truncating in $N$ the diagram (\ref{diagramaprin}):
 \begin{equation}\label{diagramaprin2}
\xymatrix{
X^G\ar@{^(->}[rr] \ar[d]&&X\ar[d]\\
F_N\ar[rd]_{\eta_N}\ar[rr]^(.52){\gamma_N}&&E_N\ar[ld]^{\xi_N}\\
&BG^{(N)}.&
}
\end{equation}
We now obtain an explicit Sullivan model of this map. For it, choose a minimal model $
\varphi\colon (\Lambda V,d)\longrightarrow (\Lambda Z,d),
$ of the inclusion $ X^G\hookrightarrow X$ and a model of diagram (\ref{diagramaprin2}) of the form
$$
\xymatrix{
&(A\otimes\Lambda V,D)\ar[r]\ar[dd]_\psi&(\Lambda V,d)\ar[dd]_\varphi\\
(A,0)\ar[ru]\ar[rd]&&\\
&(A,0)\otimes(\Lambda Z,d)\ar[r]&(\Lambda Z,d)}
$$
where
$
A=(\Lambda Q/(\Lambda Q)^{>N},0)
$ is as in (\ref{hacefalta}). By applying directly Theorem \ref{thm::modeloEspacioSecciones}, one obtain models $(\Lambda(V\otimes A^{\sharp}),\widetilde{d}\,)$ and $(\Lambda(Z\otimes A^{\sharp}),\widetilde{d}\,)$ of $\sec\xi_N$ and $\sec\eta_N$ respectively.

\begin{theorem}\label{principal2}
The composition
$$
(\Lambda(V\otimes A^{\sharp}),\widetilde{d}\,)\stackrel{\phi}{\longrightarrow} (\Lambda(Z\otimes A^{\sharp}),\widetilde{d}\,)\stackrel{\gamma}{\longrightarrow} (\Lambda Z, d)
$$
where,
$$
 \phi(v\otimes\alpha)=\rho^{-1}[\psi(v)\otimes\alpha],\quad v\otimes \alpha\in V\otimes A^{\sharp},
 $$
 and
 $$
\gamma(z\otimes\alpha)=\left\{\begin{array}{ccc}
                                     z & if & \alpha=1, \\
                                     0 & if & \alpha\neq1,
                                   \end{array}\qquad z\otimes\alpha\in Z\otimes A^\sharp,
\right.
$$
is a model of
 $$
 k\colon X^G\hookrightarrow X^{hG}.
 $$
\end{theorem}
\begin{proof} Well known facts on the rational homotopy type of mapping spaces and spaces of sections of nilpotent fibrations let us assert that $\phi$ and $\gamma$ are models of
 $$
 \map(BG^{(N)},X^G)\stackrel{{\gamma_N}_*}{\longrightarrow}\sec \xi_{N}\cong X^{hG}
 $$
and
$$
X^G\hookrightarrow \map(BG^{(N)},X^G),
$$
respectively. The composition of these maps is precisely $k$ in view of (\ref{ktruncada}).
\end{proof}

We now apply all of the above to tori actions. As before $X$ will denote a rational space with $\dim\pi_*(X)<\infty$ in which the $m$-dimensional torus $\bt$ acts. Then, a deep result of Allday \cite[Thm. 3.3]{all1} asserts that
$$
\begin{aligned}
& \dim \pi_\text{even}(X^\bt)\otimes \bq\le \dim \pi_\text{even}(X),\\
& \dim \pi_\text{odd}(X^\bt)\otimes \bq\le \dim \pi_\text{odd}(X).
\end{aligned}
$$
In particular $\dim\pi_*(X^\bt)\otimes\bq<\infty$ and we do not have to impose this extra  condition. However, we do need to assume the following \cite[Cond. 1.1]{allpu3} ó \cite[Def. 3.1.2]{allpu}: if  $\bt^0_x$ denotes the so called {\em connective isotropy group}, that is, the path component of the isotropy group $\bt_x$ of $x$ containing the unit, the set $\{\bt_x^0\}_{x\in X}$ is finite.  This condition  obviously holds for $\bs^1$-actions and it is also satisfied for $\bt$-actions of a finite CW-complex   \cite[Rem. 3.2.5]{allpu} and hence, for the induced rational action.

We whall also make use of the rational homotopy version  of the classical {\em Borel Localization Theorem} which we briefly recall.

 Let
$$
\xymatrix{
&(R\otimes\Lambda V,D)\ar[r]\ar[dd]_\psi&(\Lambda V,d)\ar[dd]_\varphi\\
R\ar[ru]\ar[rd]&&\\
&R\otimes(\Lambda Z,d)\ar[r]&(\Lambda Z,d),}
$$
be a model of the map between the non truncated Borel fibrations
$$\xymatrix{
X^\bt\ar@{^(->}[rr] \ar[d]&&X\ar[d]\\
B\bt\times X^\bt\ar[rd]_\eta\ar[rr]^(.52)\gamma&&X_{h\bt}\ar[ld]^\xi\\
&B\bt.&
}
$$
Here, $
R=\bigl(\Lambda (a_1,\dots,a_m),0)$, $|a_i|=2$, $i=1,\dots,m$. As we assume $X^{h\bt}\not=\emptyset$, there exists a retraction $(R\otimes\Lambda V,D)\to R$ and therefore, we can write
$$
Dv=\sum_{i\ge 1} D_iv,\qquad D_iv\in R\otimes \Lambda ^iV,\qquad v\in V.
$$
Since $(\Lambda V,d)$ is minimal,
$
D_1 V\subset R^+\otimes V$ and thus, defining
 $D_1$ in $R\otimes V$ by
$D_1(a\otimes v)=a D_1v$,
it turns out that
$
(R\otimes V,D_1)
$
is a chain complex which is called the {\em space of indecomposables} of the action.
The morphism induced by $\psi$ on the corresponding spaces of indecomposables is of the form,
$$
\psi_1\colon (R\otimes V,D_1)\longrightarrow (R\otimes Z,0).
$$
Localizing this morphism at
$
\bk=\bq(a_1,\dots,a_m)$,  the field of fractions of $R$, we obtain a morphism of (ungraded) differential vector spaces
$$
\overline\psi_1\colon (\bk\otimes V,D_1)\longrightarrow (\bk\otimes Z,0)=(Z_\bk,0).
$$
 If we assume $\bk$  concentrated in degree  $0$ and consider in  $V$ and $Z$ the usual $\bz_2$-grading given by the parity of the generators, $\overline\psi_1$ becomes a  morphism of degree zero. Moreover,

 \begin{theorem}\label{borelho} {\em \cite[Thm. 3.3.2]{allpu}}
The morphism
$$
\overline\psi_1\colon (\bk\otimes V,D_1)\stackrel{\simeq}{\longrightarrow} (Z_\bk,0)
$$
is a quasi-isomorphism.\hfill$\square$
\end{theorem}

In the case of $\bs^1$-actions, $R=(\Lambda a,0)$ with $|a|=2$,  a straightforward consequence of the Theorem \ref{borelho} above is:

 \begin{lemma}\label{aux} There exist a subspace $W\subset V$, with $\dim W=\dim Z$,  homogeneous basis $\{w_j\}_{j\in J}$ and $\{z_j\}_{j\in J}$ of $W$ and $Z$ respectively, and non negative integers $\{m_j\}_{j\in J}$ such that
       $$
        \psi(w_j)=a^{m_j}z_j+\Gamma_j,\qquad \Gamma_j\in R\otimes \Lambda^{\ge 2}Z,\quad j\in J.\eqno{\square}
        $$
        \end{lemma}
 Our first result involving tori actions generalizes  \cite[Thm. 4]{bufemu2}.

\begin{theorem}\label{principal3}
the morphism induced by
$$
k\colon X^\bt\longrightarrow X^{h\bt}
$$
in rational homotopy
$$
\pi_*(k)\otimes\bq\colon\ \pi_*(X^\bt)\otimes \bq\hookrightarrow \pi_*(X^{h\bt}),
$$
is injective. In particular, if $X^\bt$ and $X^{h\bt}$ are simply connected,
$
\cat (X^\bt_\bq)\le \cat X^{h\bt}.
$
\end{theorem}

\begin{remark} Observe that, whenever $X$ is the equivariant rationalization of a $\bt$-CW-complex,  $X^{\bt}$ is already a rational space and thus, theorem above asserts that
$$
\pi_*(k)\colon \pi_*(X^{\bt})\hookrightarrow\pi_*(X^{h\bt})
$$
is injective.
\end{remark}

\begin{proof} Assume first $m=1$, that is, $\bt=\bs^1$. By Theorem \ref{principal2}, the inclusion
$$
 k\colon X^{\bs^1}\hookrightarrow X^{h{\bs^1}}
 $$
is modeled by the composition
$$
(\Lambda(V\otimes A^{\sharp}),\widetilde{d}\,)\stackrel{\phi}{\longrightarrow} (\Lambda(Z\otimes A^{\sharp}),\widetilde{d}\,)\stackrel{\gamma}{\longrightarrow} (\Lambda Z, d)
$$
where,
$$
 \phi(v\otimes\alpha)=\rho^{-1}[\psi(v)\otimes\alpha],\quad v\otimes \alpha\in V\otimes A^{\sharp},
 $$
 and
 $$
\gamma(z\otimes\alpha)=\left\{\begin{array}{ccc}
                                     z & if & \alpha=1, \\
                                     0 & if & \alpha\neq1,
                                   \end{array}\qquad z\otimes\alpha\in Z\otimes A^\sharp,
\right.
$$
Here,  $A=R/R^{\ge M}$, $R=\bq[a]$ for some positive integer $M$ which can be chosen arbitrarily big.
Choose homogeneous basis $\{w_j\}_{j\in J}$ and $\{z_j\}_{j\in J}$ of $W$ and $Z$ respectively, and non negative integers $\{m_j\}_{j\in J}$ satisfying the identities of Lemma \ref{aux} above. Observe that, for each $j\in J$,
$$
\phi(w_j\otimes (a^{m_j})^\sharp)=\rho^{-1}[\psi(w_j)\otimes (a^{m_j})^\sharp]=
$$
$$
\rho^{-1}[a^{m_j}\otimes z_j\otimes (a^{m_j})^\sharp]+\rho^{-1}[\Gamma_j\otimes (a^{m_j})^\sharp]=z_j+ \Phi,\qquad \Phi\in\Lambda^{\ge2}(Z\otimes A^{\sharp}).
$$
Thus,
$$
\gamma\phi(w_j\otimes (a^{m_j})^\sharp)=z_j+ \Psi,\qquad \Psi\in \Lambda^{\ge 2}.
$$
That is, the model of $k$ is surjective on indecomposables and thus, $\pi_*(k)\otimes\bq$ is injective.

        For a general torus action one observes as in the proof of \cite[Thm. 3.7]{all1} that there exists a circle $\bs^1\subset\bt$ such that
$
X^\bt=X^{\bs^1}$.
Consider the commutative  diagram induced by this inclusion $\bs^1\subset\bt$ involving their Borel fibrations
$$\xymatrix{
X_{h\bs^1}\ar[d]\ar[r]&{X_{h\bt}}\ar[d]\\
B\bs^1\ar@/^1pc/[u]^{\overline\sigma}\ar[r]_g&B\bt,\ar@/_1pc/[u]_\sigma
}$$
and observe that, given $\sigma \in X^{h\bt}$ a section as in the diagram, the map $\sigma\circ g$ factors through  $X_{h\bs^1}$ to produce the section $\overline\sigma\in X^{h\bs^1}$. Thus, we have a map $X^{h\bt}\to X^{h\bs^1} $ closing the following square,
$$\xymatrix{
X^\bt\ar@{^(->}[r]^{k_\bt}&X^{h\bt}\ar[d]\\
X^{\bs^1}\ar@{^(->}[r]_{k_{\bs^1}}\ar@{=}[u]&X^{h\bs^1}.
}
$$
But $k_{\bs^1}$ induces an injection on  rational homotopy groups, so does $k_\bt$.

Finally, if $X^\bt$ and $X^{h\bt}$ are simply connected, the   \emph{Mapping Theorem} \cite[Thm. 28.6]{fehatho} guarantees that $\cat X^\bt_\bq\le \cat X^{h\bt}$.
\end{proof}

\begin{example}
{\em Consider in $\bs^2$ the trivial $\bs^1$-action and the $\bs^1$-action by rotation along an axis. The fixed point sets are $\bs^2$ and $\bs^0$ respectively. The models of the Borel constructions of these actions correspond respectively to choose $\lambda=0$ and $\lambda=1$ in the general model  of the Borel construction of any $\bs^1$ action on $\bs^2$ given by
$$(\Lambda a\otimes \Lambda (x,y),d)\ \text{ with }\ dy = x^2+\lambda a^{n/2} x\ \text{ and }\ |x|=2.$$

To see this, consider as above $\bk=\bq(a)$ the field of fractions of  $\Lambda a$. For $\lambda=0$, one trivially sees that $(\bk\otimes \Lambda (x,y),d)$ is the model of $\bs^2$ over $\bk$.

If $\lambda=1$ we obtain that the cohomology of $(\bk\otimes \Lambda (x,y),d)$ is $\bk[x]/(x^2-ax)$ which is isomorphic to $\bk e_1\oplus\bk e_2$ with $e_1=a^{-1}x$ and $e_2=1-a^{-1}x$. Thus, $e_1+e_2=1$, $e_1^2=e_1$ and $e_2^2=e_2$. This proves that this is a model of $S^0$ over $\bk$.
}

\end{example}

 Recall \cite[Def. 3.2]{allpu3},\cite[Def. 3.3.13]{allpu} that  the {$\bt$-rational homotopy groups}  $\pi_\bt^*(X)$ associated to a given $\bt$-action are defined as the homology of the indecomposables
$
\pi_\bt^*(X)=H^*(R\otimes V,D_1)
$. The action is said to be
{\em $\bt$-minimal} if $D_1=0$, In this case, by Theorem \ref{borelho}, $\dim V=\dim Z$. Equivalently, $
\dim\pi_*(X)=\dim\pi_*(X^\bt)\otimes\bq$.
For instance, any $\bt$-action in a space $X$ whose rational  homotopy is concentrated in odd degrees is necessarily minimal by degree reasons.

\begin{proposition}\label{nunca} For any $\bs^1$-minimal action in a non trivial,  elliptic, simply connected complex $X$, the inclusion $k\colon X^{\bs^1}\hookrightarrow X^{h\bs^1}$ is never a rational homotopy equivalence.
\end{proposition}

\begin{proof} Again, via Theorem \ref{principal2}, $k$ is modeled by
$$
(\Lambda(V\otimes A^{\sharp}),\widetilde{d}\,)\stackrel{\phi}{\longrightarrow} (\Lambda(Z\otimes A^{\sharp}),\widetilde{d}\,)\stackrel{\gamma}{\longrightarrow} (\Lambda Z, d),
$$
with $A=R/R^{\ge M}$, $R=\bq[a]$,  and  $\dim V=\dim Z$.

In the minimal model $(\Lambda V,d)$ of $X$ choose a non trivial element $v\in V$ of odd degree at least $3$ and proceed as follows: by  Lemma \ref{aux} there exist a non negative integer $n$ and  $z\in Z$ such that
$$
        \psi(v)=a^{n}z+\Gamma,\ \Gamma \in R\otimes \Lambda^{\ge 2}Z.$$ Then, if
$n\ge 1$, $\gamma\circ\phi(v\otimes 1)=0$. Otherwise, if $n=0$, choose $\alpha\in  A^\sharp$ dual of $a$. Observe that $|v\otimes\alpha|\ge 1$ and $\gamma\circ\phi(v\otimes \alpha)=0$. Therefore, $\pi_*(k)\otimes\bq$ is never surjective.
\end{proof}

\begin{proposition}\label{ejemplo} Given a $\bt$-action in the odd dimensional rational sphere $\bs^n_\bq$, the map $k\colon {\bs^n_\bq}^{\bt}\hookrightarrow  {\bs^n_\bq}^{h\bt}$ has the rational homotopy type of the inclusion of a factor
$$
\bs^{n_i}\hookrightarrow \bs^{n_1}\times\dots\times\bs^{n_p},
$$
of a product of odd dimensional spheres.
 \end{proposition}

 \begin{proof} For  $\bt=\bs^1$ the minimal model of  $ {\bs^n_\bq}^{h\bs^1}$ is computed in the proof of Theorem
 \ref{accionenesfera}(i):
 $$
 \bigl(\Lambda (v\otimes\alpha_2,v\otimes\alpha_4,\dots,v\otimes\alpha_{n-1}),0\bigr),
 $$
  with $|v|=n$, $|v\otimes\alpha_i|=n-2i$. On the other hand, in view of Theorem \ref{borelho}, and as the action is necessarily minimal, the model of  ${\bs^n_\bq}^{\bs^1}$ is $(\Lambda z,0)$, with $|z|=j$ odd no bigger than $n$.  Hence, $k$ is modeled by the morphism
  $$
  \bigl(\Lambda (v\otimes\alpha_2,v\otimes\alpha_4,\dots,v\otimes\alpha_{n-1}),0\bigr)\longrightarrow (\Lambda z,0),
  $$
  which sends $v\otimes\alpha_{\frac{n-j}{2}}$ to $z$, and is trivial on the other generators.

  For a general $\bt$-action, one follows a similar argument using Theorem
 \ref{torito}.

\end{proof}

\begin{remark}  that, for $\bt=\bs^1$, the statement above can be refined: there exists an integer $j\in\{1,3,\dots,n\}$ such that $k$ has the rational homotopy type of the inclusion
$$
\bs^j\hookrightarrow \bs^1\times\bs^3\times\dots\times\bs^n.
$$
\end{remark}

\bigskip
\bigskip

\noindent{\sc Institut de Recherche en Math\'ematique et Physique, Universit\'e Catholique de Louvain , 2 Chemin du Cyclotron B-1348, Louvain-la-Neuve, Belgium}.\hfill\break
{urtzibuijs@gmail.com}
\hfill\break
{yves.felix@uclouvain.be}
\bigskip

\noindent {\sc Facultad de Ingenier{\'\i}a, Universidad de Deusto, Avda. de las Universidades 24, 48014 Bilbao, Spain}.\hfill\break {shuerta@cyl.com}
\bigskip

\noindent {\sc Departamento de \'Algebra, Geometr\'{\i}a y Topolog\'{\i}a, Universidad de M\'alaga, Ap.\ 59, 29080 M\'alaga, Spain}.\hfill\break {aniceto@uma.es}

\end{document}